\documentclass[10pt,onecolumn]{amsart}
\newtheorem{theorem}{Theorem}[section]
\usepackage[left=1.5cm,top=1.5cm,bottom=1.5cm,right=1.5cm]{geometry}
\usepackage{latexsym}
\usepackage{amssymb,amsmath,amsfonts}

\newtheorem{lemma}[theorem]{Lemma}
\newtheorem{corollary}[theorem]{Corollary}

\theoremstyle{definition}
\newtheorem{definition}[theorem]{Definition}

\newtheorem{remark}[theorem]{Remark}

\newcommand{\R}{\mathbb R}

\newcommand{\N}{\mathbb N}

\numberwithin{equation}{section}
 \global\parskip 6pt \oddsidemargin
0.1cm \evensidemargin 0.4cm \headheight 0.1cm \topmargin -2.0cm
\begin{document}

\pagenumbering{arabic}

  \title[ Strong convergence theorems for strongly monotone mappings]
 {Strong convergence theorems for strongly monotone mappings in Banach spaces}
\author{ M.O. Aibinu$^{1}$, O.T. Mewomo$^2$}
\address{$^{1, 2}$ School of Mathematics, Statistics and Computer Science, University of KwaZulu-Natal, Durban, South Africa.}
\address{$^1$ DST-NRF Center of Excellence in Mathematical and Statistical Sciences (CoE-MaSS).}
\email{ }
 \email{  $^1$moaibinu@yahoo.com, $^1$216040407@stu.ukzn.ac.za}
\email{$^2$mewomoo@ukzn.ac.za }

 \keywords{ Range condition, Strongly monotone, Lyapunov function, Strong convergence.\\
{\rm 2010} {\it Mathematics Subject Classification}: 47H06, 47H09, 47J05, 47J25.\\
{\bf This manuscript should be cited as}: M.O. Aibinu, O.T. Mewomo, Strong convergence theorems for strongly monotone mappings in Banach spaces, Boletim da Sociedade Paranaense de Matemática, (2018), DOI:10.5269/bspm.37655}

\begin{abstract}
Let $E$ be a uniformly smooth and uniformly convex real Banach space and $E^*$ be its dual space. Suppose $A : E\rightarrow  E^*$ is bounded, strongly monotone and satisfies the range condition such that $A^{-1}(0)\neq \emptyset$. Inspired by Alber \cite{b1}, we introduce Lyapunov functions and use the new geometric properties of Banach spaces to show the strong convergence of an iterative algorithm to the solution of $Ax=0$.
\end{abstract}

\maketitle

\section{Introduction}
\par Let $H$ be a real Hilbert space. A mapping $A : D(A)\subset H\rightarrow H$ is said to be monotone if for every $x, y \in D(A)$, we have $$\left\langle x - y, Ax-Ay\right\rangle \geq 0.$$ 
$A$ is called maximal monotone if it is monotone and the range of $(I+tA)$ is all of $H$ for some $t>0$. Consider the following problem:
 \begin{equation}\label{h1}
\mbox{find u} \ \in H \ \mbox{such that} \  0 \in Au,
\end{equation}
where $A$ is a maximal monotone mapping on $H$. This is a typical way of formulating many problems in nonlinear analysis and optimization. A well-known method for solving (\ref{h1}) in a Hilbert space is the proximal point algorithm: $x_1 \in H$ and 
$$x_{n+1} = J_{r_n}x_n,\ \ n\in \N,$$
 introduced by Martinet \cite{b36} and studied further by Rockafellar \cite{b5} and a host of other authors. Monotone mappings were first studied in Hilbert spaces by Zarantonello \cite{b38}, Minty \cite{b39}, Ka$\check{c}$urovskii \cite{b40} and a host of other authors.  Interest in monotone mappings stems mainly from their usefulness in numerous applications. Consider for example (see e.g Chidume et al. \cite{b43}), the following: Let $f : E \rightarrow \R$ be a proper and convex function. The subdifferential of $f$ at $x\in E$ is defined by 
$$\partial f(x)=\left\{x^*\in E^*: f(y)-f(x)\geq\left\langle y-x, x^*  \right\rangle \forall ~~y\in E \right\}.$$
Monotonicity of $\partial f : E \rightarrow 2^{E^*}$ on $E$ can be easily verified, and that $0\in \partial f(x)$
if and only if $x$ is a minimizer of $f.$  Setting $\partial f = A$, it follows that solving the inclusion $0\in Au$ in this case, is the same as solving for a minimizer of $f$.  Several existence theorems have been established for the equation $Au = 0$ when $A$ is of the monotone-type (see e.g., Deimling \cite{b41}; Pascali and Sburlan \cite{b11}). Let $E$ be a real normed space and let $J_p, (p > 1)$ denote the generalized duality mapping from $E$ into $2^{E^*}$ given by 
$$J_p(x) =\left\{ f \in E^* :\left\langle x, f \right\rangle = {\|x\|}^p, \| f\|={\| x\|}^{p-1} \right\},$$
where $E^*$ denotes its dual space and $\left\langle . , .\right\rangle$, the generalized duality pairing. If $E$ is a uniformly smooth Banach space with $J_{p} : E \rightarrow E^*$ and $J_{q}^* : E^* \rightarrow E$ being the duality mappings with gauge functions $\nu(t)=t^{p-1}$ and $\nu(s)=s^{q-1}$ respectively, then $J_{p}^{-1}=J_{q}^*$. For $p = 2$, the mapping $J_2$ from $E$ to $2^{E^*}$ is called normalized duality mapping. If there is no danger of confusion, we omit the subscript $p$ of $J_p$ and simply write $J$. If $E$ is smooth, then $J$ is single-valued and onto if $E$ is reflexive (see e.g., Alber and Ryazantseva \cite{b2}, p. 36, Cioranescu \cite{b3}, p. 25-77, Xu and Roach \cite{b7}, Z$\check{a}$linescu  \cite{b37}). 

Let $X$ and $Y$ be real normed linear spaces and $f :U\subset X\rightarrow Y$ be a map with U open and nonempty.  The function $f$ is said to have a G$\hat{a}$teaux derivative at $u\in U$ if there exists a bounded linear map from $X$ into $Y$ denoted by $D_Gf(u)$ such that for each $h$ in $X$, we have 
\begin{equation} \label{b33}
\underset{t\rightarrow 0}\lim \frac{f(u+th)-fu}{t}=\left\langle D_Gf(u),h\right\rangle.
\end{equation}
We say that $f$ is G$\hat{a}$teaux differentiable if it has a G$\hat{a}$teaux derivative at each $u$ in $U$. Let $X$ and $Y$ be real Banach spaces. A mapping $A : D(A) \subset X \rightarrow Y$ is said to be uniformly continuous if for all $\epsilon >  0,$ there exists $\delta >0$ such that for all $x, y \in D(A),~~{\|x - y\|}_X < \delta~~\Rightarrow {\|Ax - Ay\|}_Y<\epsilon$. A function $\psi: [0,\infty)\rightarrow [0,\infty)$ such that $\psi$ is nondecreasing, $\psi(0) = 0$ and $\psi$ is continuous at $0$ is called a modulus of continuity. It follows that $A$ is uniformly continuous if and only if it has a modulus of continuity and $A$ is said to be $\alpha$-H$\ddot{o}$lder continuous if for some $0<\alpha \leq 1$, there exists a positive constant $k$ such that $\psi(t)\leq kt^{\alpha}$ for all $t\in [0,\infty)$. Let $E$ be a smooth Banach space, the single-valued mapping $A :E \rightarrow E^*$
\begin{itemize}
	\item [(i)] is monotone if for each $x,y\in E$, we have $${\left\langle  x - y, Ax - Ay \right\rangle} \geq 0;$$
  \item [(ii)]is $\eta$-strongly monotone if there exist a constant $\eta > 0$ such that for each $x,y\in E$, we have $${\left\langle  x - y, Ax - Ay \right\rangle} \geq \eta {\|x-y \|}^p;$$
  \item [(iii)]is maximum monotone if it is monotone and the range of $(J+tA)$ is all of $E^*$ for some $t>0$;
  \item [(iv)]satisfies the range condition if it is monotone and the range of $(J+tA)$ is all of $E^*$ for all $t>0$.
\end{itemize}
\begin{remark}
Observe that any maximal monotone mapping satisfies the range condition. The converse is not necessarily true. Hence, range condition is weaker than maximal monotone.
\end{remark}
Let  $A :E \rightarrow E$ be a single-valued mapping.
\begin{itemize}
	\item [(i)]$A$ is accretive if for each $x, y \in E$, there exists $j(x-y)\in J(x-y)$ such that\\ $\left\langle j(x-y), Ax-Ay\right\rangle \geq 0;$ 
	\item [(ii)] $A$ is $\eta$-strongly accretive if for each $x, y \in E$, there exists $j(x-y)\in J(x-y)$ and a constant $\eta > 0$ such that $\left\langle j(x-y), Ax-Ay\right\rangle \geq \eta {\|x-y \|}^p;$
  \item [(iii)]$A$ is m-accretive if it is accretive and the range of $(I+tA)$ is all of $E$ for some $t>0$;
  \item [(iv)]$A$ satisfies the range condition if it is accretive and the range of $(I+tA)$ is all of $E$ for all $t>0$.
\end{itemize}
\begin{remark}\label{h3}
 Chidume and Djitte \cite{b6}. For a real $q > 1$, let $E$ be a $q$-uniformly smooth real Banach space and $A : E \rightarrow E$ be a map
with $D(A)= E$. Suppose that $A$ is m-accretive, then $A$ satisfies the range condition.
\end{remark}
However, the converse is not necessarily true. Hence, range condition is weaker than m-accretive. In a Hilbert space, the normalized duality map is the identity map. Hence, in Hilbert spaces, monotonicity and accretivity coincide.
\par There have been extensive research efforts on inequalities in Banach spaces and their applications to iterative methods for solutions of nonlinear equations of the form $Au=0.$ Assuming existence, for approximating a solution of $Au = 0$, where $A$ is of accretive-type, Browder \cite{b44} defined an operator $T : E \rightarrow E$ by $T := I - A$, where $I$ is the identity map on $E$. He called such an operator pseudo-contractive. It is trivial to observe that zeros of $A$ correspond to fixed points of $T$. For Lipschitz strongly pseudo-contractive maps, Chidume \cite{b22} proved the following theorem. 
\begin{theorem}\label{t5}
Chidume \cite{b22}.  Let $E = L_p, 2 \leq p < \infty$, and $K \subset E$ be nonempty closed convex and bounded. Let $T : K \rightarrow K$ be a strongly pseudocontractive and Lipschitz map. For arbitrary $x_1 \in K$, let a sequence $\left\{x_n\right\}$ be defined iteratively by $x_{n+1} = (1 - {\lambda}_n)x_n + {\lambda}_nTx_n, n\in \N$, where ${\lambda}_n \in (0, 1)$ satisfies the following conditions:
 \begin{itemize}
\item[(i)] $\displaystyle \sum^{\infty}_{n=1}{\lambda}_n = \infty$,
\item[(ii)] $\displaystyle \sum^{\infty}_{n=1}{\lambda}^2_n < \infty$.
 \end{itemize}  
 Then, $\left\{x_n\right\}$ converges strongly to the unique fixed point of $T$.
\end{theorem}
The above theorem has been generalized and extended in various directions, leading to flourishing areas of research, for the past forty years or so, for numerous authors (see e.g., Censor and Reich \cite{b23}; Chidume \cite{b25}, \cite{b22}, \cite{b6}; Chidume and
Bashir \cite{b27}; Chidume and Chidume \cite{b29}; Chidume and Osilike \cite{b30} and a host of other authors).
\par  However, it occurs that most of the existing results on the approximation of solutions of monotone-type mappings have been proved in Hilbert spaces or they are for accretive-type mappings in Banach spaces. Unfortunately, as has been rightly observed, many and probably most mathematical objects and models do not naturally live in Hilbert spaces. The remarkable success  in approximating the zeros of accretive-type mappings is yet to be carried over to equations involving nonlinear monotone mappings in general Banach spaces.  Perhaps, part of the difficulty in extending the existing results on the approximation of solutions of accretive-type mappings to general Banach spaces is that, since the operator $A$ maps $E$ to $E^*$, the recursion formulas used for accretive-type mappings may no longer make sense under these settings. Take for instance, if $x_n$ is in $E$, $Ax_n$ is in $E^*$ and any convex combination of $x_n$ and $Ax_n$ may not make sense. Moreover, most of the inequalities used in proving convergence theorems when the operators are of accretive-type involve the normalized duality mappings which also appear in the definition of accretive operators. 

\par Alber \cite{b1} introduced a Lyapunov functional which signaled the beginning of the development of new geometric properties in Banach spaces. The Lyapunov function introduced by Alber is suitable for studying iterative methods for approximating solutions of equation $0\in Au$ where $A : E\rightarrow 2^{E^*}$ is of monotone type and other related problems (see e.g \cite{b9}, \cite{b43}, \cite{f3}, \cite{f4}). Inspired by Alber \cite{b1}, our purpose in this paper is to use the new geometric properties to study an iterative scheme for the strongly monotone mappings. Therefore, we introduce Lyapunov functions and prove the strong convergence theorem for strongly monotone mappings in uniformly smooth and uniformly convex Banach spaces.

\section{Preliminaries}
Let $E$ be a real normed space of dimension $\geq$ 2  and let $S := \left\{x \in E : \|x \| = 1\right\}$. $E$ is said to have a G$\hat{a}$teaux differentiable norm (or $E$ is called smooth) if the limit
$$\displaystyle \lim_{t\rightarrow 0}\frac{\|x +ty \|-\|x \|}{t}$$
exists for each $x, y \in S$; E is Fr$\acute{e}$chet differentiable if it is smooth and the limit is attained uniformly for $y\in S$.
 Further, $E$ is said to be uniformly smooth if it is smooth and the limit is attained uniformly for each $x, y \in S$.
 The modulus of convexity of $E$, $\delta_{E}: (0, 2]\rightarrow [0, 1]$ is defined by
 $$\delta_{E}(\epsilon)=\inf \left\{1-\frac{\|x + y \|}{2}  : \|x \|=\|y \|=1, \|x - y \| > \epsilon \right\}.$$
$E$ is uniformly convex if and only if ${\delta}_E(\epsilon) > 0$ for every $\epsilon \in(0, 2]$.
 Let $p > 1$, then $E$ is said to be $p$-uniformly convex if there exists a constant $c > 0$ such that ${\delta}_E(\epsilon) \geq c{\epsilon}^p$ for all $\epsilon \in (0, 2]$. Observe that every $q$-uniformly smooth space is uniformly smooth and every $p$-uniformly convex space is uniformly convex. A normed linear space $E$ is said to be strictly convex if
$$\|x \|=\|y \|=1, x\neq y\Rightarrow \frac{\|x + y \|}{2}<1.$$
Every uniformly convex space is strictly convex. 
Typical examples of such spaces, (see e.g., Chidume \cite{b8}, p. 34, 54) are the $L_p, l_p$, and $W_p^m$ spaces for $1 < p < \infty$, where
\begin{equation}
\rho_{L_p}(\tau)=\rho_{l_p}(\tau)=\rho_{W_p^m}(\tau) \leq \left \{ \begin{array}{cl} \frac{1}{p}{\tau}^p & {\rm}\ 1<p<2,\\
 \frac{p-1}{2}{\tau}^2  & {\rm}\ 2\leq p <\infty ,\end{array} \right.
\end{equation}
and
\begin{equation}
\delta_{L_p}(\epsilon)= \delta_{l_p}(\epsilon)=\delta_{W_p^m}(\epsilon) \geq \left \{ \begin{array}{cl} \frac{1}{2^{p+1}}{\epsilon}^2 & {\rm}\ 1<p<2,\\
 {\epsilon}^p  & {\rm}\ 2\leq p <\infty .\end{array} \right.
\end{equation}
\begin{definition}\label{d2}
Let $E$ be a smooth real Banach space with the dual $E^*$.
\begin{itemize}
	\item[(i)] The function ${\phi} : E \times E \rightarrow \R$ is defined by 
\begin{equation}\label{e1}
{\phi} (x,y)= {\|x \|}^2 - 2\left\langle x, J_2(y) \right\rangle + {\|y \|}^2, \ \mbox{for all} \ x,y \in E,
\end{equation}
where $J_2$ is the normalized duality map from $E$ to $E^*$ introduced by Alber and has been studied by Alber \cite{b1}, Kamimura and Takahashi \cite{b10} and Reich \cite{r17}. 
\item[(ii)]The map $V: E\times E^* \rightarrow \R $ is defined by 
$$V (x,x^*)= {\|x \|}^2 - 2\left\langle x, x^* \right\rangle + {\|x^* \|}^2  ~~ \forall ~~ x \in E, x^*\in E^*.$$
\end{itemize}	
\end{definition}

If $E = H,$ a real Hilbert space, then Eq.(\ref{e1}) reduces to ${\phi} (x,y)= {\|x-y \|}^2$ for $x, y \in H$.\\
Also, it is obvious from the definition of the function ${\phi}$ that
\begin{equation}
(\|x\|- \|y \|)^2 \leq {\phi} (x,y) \leq (\|x\| + \|y \|)^2  \ \mbox{for all} \ x,y \in E.
\end{equation}

\par We need the following lemmas and theorems in the sequel.
 \begin{lemma}\label{l13}
B. T. Kien \cite{b31}. The dual space $E^*$ of a Banach space $E$ is uniformly convex if and only if the duality mapping $J_p$ is a single-valued map which is uniformly continuous on each bounded subset of $E$.
 \end{lemma}
 
\begin{lemma}\label{l12}
Z$\check{a}$linescu  \cite{b37}. Let $\psi : {\R}^+\rightarrow {\R}^+$ be increasing with $\displaystyle \lim_{t\rightarrow \infty} \psi(t)=\infty.$ Then $J^{-1}_{\psi}$ is single-valued and uniformly continuous on bounded sets of $E^*$ if and only if $E$ is a uniformly convex Banach space.
\end{lemma}

\begin{theorem}\label{t3}
Xu \cite{b18}. Let $E$ be a real uniformly convex Banach space. For arbitrary $r>0$, let $B_r(0):=\left\{x \in E: \|x\| \leq r\right\}$. Then, there exists a continuous strictly increasing convex function
$$g:[0,\infty)\rightarrow [0,\infty),~~ g(0)=0,$$
such that for every $x, y \in B_r(0), j_p(x)\in J_p(x), j_p(y)\in J_p(y)$, the following inequalities hold:
 \begin{itemize}
	\item[(i)]${\|x + y \|}^p \geq {\|x \|}^p + p\left\langle y, j_p(x) \right\rangle + g(\|y \|)$;
	\item[(ii)]$\left\langle x-y, j_p(x)-j_p(y) \right\rangle \geq g(\|x-y \|)$.
 \end{itemize}	
\end{theorem}

\begin{lemma}\label{l11}
 Xu \cite{bh1}. Let $\left\{a_n\right\}$ be a sequence of nonnegative real numbers satisfying the following
relations:
$$a_{n+1}\leq(1-{\alpha}_n)a_n+{\alpha}_n{\sigma}_n+{\gamma}_n, ~~n\in \N,$$
where
\begin{itemize}
\item[(i)]${\left\{\alpha\right\}}_n\subset (0, 1)$, $\displaystyle\sum_{n=1}^{\infty} {\alpha}_n = \infty$;
\item[(ii)]$\limsup {\left\{\sigma\right\}}_n\leq 0$;
\item[(iii)] ${\gamma}_n\geq 0$, $\displaystyle\sum_{n=1}^{\infty} {\gamma}_n = \infty$.
\end{itemize}
 Then, $a_n\rightarrow 0$ as $n\rightarrow \infty$.
\end{lemma}

\begin{theorem}\label{t8}
Kido \cite{b30}. Let $E^*$ be a real strictly convex dual Banach space with a Fr$\acute{e}$chet differentiable norm and $A$ a maximal monotone operator from $E$ into $E^*$ such that $A^{-1}0\neq \emptyset$. Let $J_tx:=(J+tA)^{-1}x$ be the resolvent of $A$ and $P$ be the nearest point retraction of $E$ onto $A^{-1}0$. Then, for every $x\in E$, $J_tx$ converges strongly to $Px$ as $t\rightarrow \infty$.
\end{theorem}

\begin{lemma}\label{l10}
Kamimura and Takahashi \cite{b10}. Let $E$ be a smooth uniformly convex real Banach space and let $\left\{x_n\right\}$ and $\left\{y_n\right\}$ be two sequences from $E.$ If either $\left\{x_n\right\}$ or $\left\{y_n\right\}$ is bounded and $\phi (x_n, y_n) \rightarrow 0$ as $n \rightarrow \infty$, then $ \| x_n - y_n \| \rightarrow 0$ as $n \rightarrow \infty$.
\end{lemma}
\begin{lemma}\label{h10}
Alber and Ryazantseva \cite{b2}, p. 17. If a functional $\phi$ on the open convex set $M \subset$ dom $\phi$ has a subdifferential,
then $\phi$ is convex and lower semicontinuous on the set.
\end{lemma}
\begin{lemma}\label{h9}
 Rockafellar \cite{b5}. Let $E$ be a reflexive smooth Banach space and let $A$ be a monotone operator from $E$ to $E^*$. Then $A$ is maximal if and only if $R(J + rA) = E^*$ for all $r > 0$. That is, every maximal monotone map satisfies the range condition.
\end{lemma}
\begin{lemma}\label{h4}
 Cioranescu \cite{b3}, p. 156. Let $A: X\rightarrow X^*$ be a semicontinuous monotone mapping with $D(A)=X$. Then $A$ is maximal monotone.
\end{lemma}
\begin{lemma}\label{l15}
Chidume and Idu \cite{b28}. Let $E$ be an arbitrary real normed space and $E^*$ be its dual space. Let $A : E\rightarrow 2^{E^*}$ be any mapping. Then $A$ is monotone if and only if $T := (J-A) : E \rightarrow 2^{E^*}$ is $J$-pseudocontractive. 
\end{lemma}
\begin{lemma}\label{l25}
(See, e.g., Chidume and  Djitte \cite{b6}).  Let $X$ and $Y$ be real normed linear spaces and let  $T : X \rightarrow Y$ be a uniformly continuous map. For arbitrary $r > 0$ and fixed $x^* \in X$, let $$B_X(x^*, r) : \left\{x \in X : {\|	x - x^*\|}_X \leq r\right\}.$$ 
Then $T\left(B(x^*, r)\right)$ is bounded.
\end{lemma}

\section{Main Results}
We first give give some definitions and prove the lammas which are useful in establishing our main results.
\begin{definition}\label{d1}
Let $E$ be a smooth real Banach space with the dual $E^*$.
\begin{itemize}
\item[(i)] We introduce the function ${\phi}_p : E \times E \rightarrow \R$ defined by
$${\phi}_p (x,y)=  \frac{p}{q}{\|x \|}^q - p\left\langle x, Jy \right\rangle + {\|y \|}^p , \ \mbox{for all} \ x,y \in E,$$
where $J$ is the generalized duality map from $E$ to $E^*$, $p$ and $q$ are real numbers such that $q\geq p>1$ and $\frac{1}{p} + \frac{1}{q} =1$.
\item[(ii)] We introduce the function $V_p: E\times E^* \rightarrow \R $ defined as 
$$V_p (x,x^*)= \frac{p}{q}{\|x \|}^q - p\left\langle x, x^* \right\rangle + {\|x^* \|}^p  ~~ \forall ~~ x \in E, x^*\in E^* \ \mbox{such that} \ q\geq p>1,\ \ \frac{1}{p} + \frac{1}{q} =1.$$
\end{itemize}	
\end{definition}

\begin{remark}\label{r1}
These remarks follow from Definition \ref{d1}:
\begin{itemize}
\item[(i)]  For $p = 2, {\phi}_2 (x,y)= {\phi} (x,y)$. Also, it is obvious from the definition of the function ${\phi}_p$ that
\begin{equation}\label{e11}
(\|x\|- \|y \|)^p \leq {\phi}_p (x,y) \leq (\|x\| + \|y \|)^p \ \mbox{for all} \ x,y \in E.
\end{equation}
	\item [(ii)] It is obvious that
\begin{equation}\label{e20}
V_p (x,x^*)= {\phi}_p (x,J^{-1}x^*)~~ \forall ~~ x \in E, x^*\in E^*.
\end{equation}
\end{itemize}
\end{remark}

\begin{lemma}\label{l8}
Let $E$ be a smooth uniformly convex real Banach space. For $d > 0$, let $ B_d(0):= \left\{ x \in E: \| x \| \leq d \right\} $. Then for arbitrary $x, y \in B_d(0)$,
$$ {\|x-y \|}^p \geq{\phi}_p (x,y)- \frac{p}{q}{\| x\|}^q,~~ q\geq p>1,~~ \frac{1}{p}+\frac{1}{q}=1.$$
\end{lemma}

\begin{proof}
 Since $E$ is a uniformly convex space, then by Theorem \ref{t3}, we have for arbitrary $x,y \in B_d(0)$,
 $$ {\|x + y \|}^p \geq {\|x \|}^p + p\left\langle y, Jx \right\rangle + g(\|y \|).$$ 
 Replacing $y$ by $-y$ gives
 $$ {\|x - y \|}^p \geq {\|x \|}^p - p\left\langle y, Jx \right\rangle + g(\|y \|).$$ 
Interchanging $x$ and $y$, we have 
\begin{eqnarray*}
 {\|x - y \|}^p & \geq & {\|y \|}^p - p\left\langle x, Jy \right\rangle + g(\|x \|)\\
 &\geq & \frac{p}{q}{\| x\|}^q - p\left\langle x, Jy \right\rangle + {\|y \|}^p-\frac{p}{q}{\| x\|}^q + g(\|x \|)\\
 &\geq & {\phi}_p (x,y)- \frac{p}{q}{\| x\|}^q + g(\|x \|)\\
  &\geq & {\phi}_p (x,y)- \frac{p}{q}{\| x\|}^q.
 \end{eqnarray*}   
\end{proof}

\begin{lemma}\label{l20}
 Let $E$ be a smooth uniformly convex real Banach space with $E^*$ as its dual. Then,
\begin{equation}\label{e21}
V_p (x, x^*) + p\left\langle J^{-1}x^*-x, y^* \right\rangle \leq V_p(x, x^*+y^*)
\end{equation}
for all $x\in E$ and $x^*, y^* \in E^*$.
\end{lemma}
\begin{proof}
\begin{eqnarray}
V_p (x,x^*)& =&  \frac{p}{q}{\|x \|}^q  -  p\left\langle x, x^* \right\rangle + {\|x^* \|}^p, \nonumber\\
V_p (x,x^*+y^*) &=&  \frac{p}{q}{\|x \|}^q - p\left\langle x, x^*+y^* \right\rangle + {\|x^*+y^* \|}^p. \nonumber \\
V_p (x, x^*+y^*) - V_p (x,x^*) &=& - p\left\langle x, y^* \right\rangle + {\|x^*+y^* \|}^p-{\|x^* \|}^p \nonumber\\
&\geq& p\left\langle -x, y^* \right\rangle+ {\|x^* \|}^p+ p\left\langle y^*, J^{-1}x^*\right\rangle+g(\|y^* \|)-{\|x^* \|}^p\nonumber\\
&& (\ \mbox{by Theorem \ref{t3}})\nonumber\\
&\geq& p\left\langle J^{-1}x^*-x^*, y^* \right\rangle,
\end{eqnarray}
so that
$$V_p (x,x^*)+ p\left\langle J^{-1}x^*-x^*, y^* \right\rangle \leq  V_p (x, x^*+y^*).$$
\end{proof}
\begin{lemma}\label{l31}
 Let $E$ be a reflexive strictly convex and smooth real Banach space with the dual $E^*$. Then,
\begin{equation}\label{e31}
{\phi}_p(y,x)-{\phi}_p(y,z)\geq  p\left\langle z-y, Jx-Jz\right\rangle= p\left\langle y-z, Jz-Jx\right\rangle  \ \mbox{ for all} \ x, y, z\in E.
\end{equation}
\end{lemma}
\begin{proof}
Consider the functional $V_p: E\times E^*\rightarrow \R$ with respect to the variable $y^*$ and a fixed element $y$ defined as
$$V_p (y,y^*)=\frac{p}{q}{\| y\|}^q - p\left\langle y, y^* \right\rangle + {\|y^* \|}^p \ \mbox{ for all} \ y \in E,  y^*\in E^*.$$
We first show that $V_p$ has a subdifferential on open subset $M\subset $ dom $V_p$. For every $h\in E^*$ and $t\in {\R}\backslash \left\{0 \right\}$, we have,
\begin{eqnarray*}
V_p (y,y^*)&=&\frac{p}{q}{\| y\|}^q - p\left\langle y, y^* \right\rangle + {\|y^* \|}^p,\\
V_p (y,y^*+th)&=&\frac{p}{q}{\| y\|}^q - p\left\langle y, y^*+th \right\rangle + {\|y^*+th \|}^p\\
&\geq& \frac{p}{q}{\|y \|}^q -p\left\langle y, y^*\right\rangle-pt\left\langle y, h \right\rangle+ {\| y^*\|}^p+pt\left\langle J^{-1}y^*, h\right\rangle + g(\| th\|), \\
\mbox{then} \ \displaystyle \lim_{t\rightarrow 0}\frac{V_p (y,y^*+th)-V_p (y,y^*)}{t} &\geq& p\left\langle J^{-1}y^*-y, h\right\rangle.
\end{eqnarray*}
Therefore, grad $V_p (x,y)=p(J^{-1}y^*-y)$ and by the Lemma \ref{h10}, $V_p$ is convex and lower semicontinuous. Then it follows from the definition of subdifferential that
$$V_p (y,x^*)-V_p(y,z^*)\geq p\left\langle J^{-1}z^*-y, x^*-z^*\right\rangle  \ \mbox{ for all} \ y \in E, x^*, z^* \in E^*.$$
Since ${\phi}_p (y,x)=V_p (y,J^{-1}x^*)$, we have
$${\phi}_p(y,x)-{\phi}_p(y,z)\geq p\left\langle z-y, Jx-Jz\right\rangle  \ \mbox{ for all} \ x, y, z\in E.$$
\end{proof}

\begin{theorem}\label{t6}
Let $E$ be a uniformly smooth and uniformly convex real Banach space and $E^*$ be its dual space. Suppose $A : E\rightarrow  E^*$ is bounded, $\eta$-strongly monotone and satisfies the range condition such that $A^{-1}(0)\neq \emptyset$. Let  $\left\{ {\lambda}_n\right\}$ and  $\left\{ {\theta}_n\right\}$ be real sequences in $(0,1)$ such that,
\begin{itemize}
	\item[(i)] $\lim {\theta}_n =0$ and $\left\{ {\theta}_n\right\}$ is decreasing;
	\item[(ii)]$ \displaystyle\sum_{n=1}^{\infty} {\lambda}_n {\theta}_n=\infty$;
	\item[(iii)]$\displaystyle \lim_{n\rightarrow \infty}\left(({\theta}_{n-1}/{\theta}_n)-1\right)/{\lambda}_n {\theta}_n=0,~~ \displaystyle\sum_{n=1}^{\infty} {\lambda}^2_n <\infty$.
\end{itemize}		
 For arbitrary $x_1 \in E$, define $\left\{x_n\right\}$ iteratively by:
\begin{equation}
x_{n+1} = J^{-1}\left(Jx_n - {\lambda}_n\left(Ax_n+{\theta}_n(Jx_n-Jx_1)\right)\right), n \in \N,
\end{equation}
where $J$ is the generalized duality mapping from $E$ into $E^*$.
There exists a real constant ${\epsilon}_0>0$ such that ${\psi}({\lambda}_nM_0)\leq{\epsilon}_0, \ \ n \in \N$ for some constant $M_0>0$. Then, the sequence $\left\{x_n \right\}$ converges strongly to the solution of $Ax=0$. 
\end{theorem}

\begin{proof}
  Let $x^*\in E$ be a solution of the equation $Ax=0$. There exists $r > 0$ sufficiently large such that:
\begin{equation}\label{e22}
r\geq \max \left\{ {\phi}_p( x^*, x_1), \frac{4p}{q}{\| x^*\|}^q \right\}.
\end{equation}
 We divide the proof into two parts.\\
{\bf Part 1:} We prove that $\left\{x_n \right\}$ is bounded. It suffices to show that ${\phi}_p(x^*, x_n)\leq r, \forall ~~ n\in \N.$ The proof is by induction. By construction, ${\phi}_p(x^*, x_1)\leq r$.
 Suppose that ${\phi}_p(x^*, x_n)\leq r$ for some $n\in \N$. We show that ${\phi}_p(x^*, x_{n+1})\leq r.$ \\
 From inequality (\ref{e11}), for real $p>1$, we have $\|x_n \| \leq r^{\frac{1}{p}} + \|x^* \|$. Since $A$ is bounded and by Lemma \ref{l13}, $J$ is uniformly continuous on bounded subsets of $E$, we define
\begin{equation}\label{e23}
M_0 :=\sup \left\{ {\|Ax_n+{\theta}_n(Jx_n-Jx_1)\|}: {\theta}_n\in(0,1), \|x_n \| \leq r^{\frac{1}{p}} + \|x^* \|  \right\}+1.
\end{equation}
Let ${\psi}$ be the modulus of continuity of $J_p^{-1} : E^*\rightarrow E$ on bounded subsets of $E^*$. Observe that by the uniform continuity of $J^{-1}$ on bounded subsets of $E^*$, we have
\begin{eqnarray}\label{e30}
\|x_n-  J^{-1}( Jx_n - {\lambda}_n\left(Ax_n+{\theta}_n(Jx_n-Jx_1)\right)) \|&=&\|J^{-1}(Jx_n)-  J^{-1}( Jx_n - {\lambda}_n\left(Ax_n+{\theta}_n(Jx_n-Jx_1)\right)) \| \nonumber\\
&\leq&{\psi}({\lambda}_nM_0).
\end{eqnarray}
Define $${\epsilon}_0:=\min\left\{1,\frac{\eta r}{4M_0}\right\} \ \mbox{where} \ {\psi}({\lambda}_nM_0)\leq{\epsilon}_0.$$
Applying Lemma \ref{l20} with $y^* := {\lambda}_n \left(Ax_n+{\theta}_n(Jx_n-Jx_1)\right)$ and by using the definition of $x_{n+1}$, we compute as follows,
 \begin{eqnarray*}
{\phi}_p(x^*, x_{n+1})
 & = & {\phi}_p\left(x^*, J^{-1}\left(Jx_n - {\lambda}_n\left(Ax_n+{\theta}_n(Jx_n-Jx_1)\right)\right)\right)\\
 & = & V_p\left(x^*, Jx_n - {\lambda}_n\left(Ax_n+{\theta}_n(Jx_n-Jx_1)\right)\right)\\
 & \leq & V_p(x^*, Jx_n )\\
 &&-p{\lambda}_n\left\langle J^{-1}( Jx_n - {\lambda}_n\left(Ax_n+{\theta}_n(Jx_n-Jx_1)\right))-x^*, Ax_n+{\theta}_n(Jx_n-Jx_1)\right\rangle \\
 &= & {\phi}_p(x^*, x_n )-p{\lambda}_n \left\langle x_n-x^*,Ax_n+{\theta}_n(Jx_n-Jx_1)\right\rangle\\
 & &-p{\lambda}_n\left\langle J^{-1}( Jx_n - {\lambda}_n\left(Ax_n+{\theta}_n(Jx_n-Jx_1)\right))-x_n, Ax_n+{\theta}_n(Jx_n-Jx_1)\right\rangle.
 \end{eqnarray*}
 By Schwartz inequality and uniform continuity of $J^{-1}$ on bounded subsets of $E^*$ (Lemma \ref{l12}), we obtain
\begin{eqnarray*}
{\phi}_p(x^*, x_{n+1})
& \leq & {\phi}_p(x^*, x_n )-p{\lambda}_n \left\langle x_n-x^*,Ax_n+{\theta}_n(Jx_n-Jx_1)\right\rangle \\
& & + p{\lambda}_n{\psi}({\lambda}_nM_0)M_0 \ \mbox{(By applying inequality (\ref{e30}))}\\
& \leq &{\phi}_p(x^*, x_n )-p{\lambda}_n \left\langle x_n-x^*,Ax_n-Ax^*\right\rangle  (\mbox{since ~} \ x^*\in N(A)) \\
& &-p{\lambda}_n{\theta}_n \left\langle x_n-x^*, Jx_n-Jx_1\right\rangle +p{\lambda}_n{\psi}({\lambda}_nM_0)M_0.
\end{eqnarray*}
By Lemma \ref{l31}, $p\left\langle x^*-x_n, Jx_n-Jx_1\right\rangle \leq {\phi}_p(x^*, x_n )-{\phi}_p(x^*, x_1 )=0$. Therefore, using strong monotonicity property of $A,$ we have,
\begin{eqnarray*}
{\phi}_p(x^*, x_{n+1})
 &\leq&  {\phi}_p(x^*, x_n )-p\eta{\lambda}_n{\|x_n -x^* \|}^p -p{\lambda}_n{\theta}_n \left\langle x_n-x^*, Jx_n-Jx_1\right\rangle +p{\lambda}_n{\psi}({\lambda}_nM_0)M_0 \\
 &\leq&  {\phi}_p(x^*, x_n )-p\eta{\lambda}_n{\|x_n -x^* \|}^p +p{\lambda}_n{\theta}_n \left\langle x^*-x_n, Jx_n-Jx_1\right\rangle +p{\lambda}_n{\psi}({\lambda}_nM_0)M_0  \\
 &\leq&  {\phi}_p(x^*, x_n )-p\eta{\lambda}_n\left({\phi}_p (x^*, x_n)- \frac{p}{q}{\| x^*\|}^q \right) +p{\lambda}_n{\psi}({\lambda}_nM_0)M_0 \\
 &=&  {\phi}_p(x^*, x_n )-p\eta{\lambda}_n{\phi}_p(x^*, x_n )+ p\eta{\lambda}_n\left(\frac{p}{q}{\| x^*\|}^q \right)+p{\lambda}_n{\psi}({\lambda}_nM_0)M_0 \\
 &\leq& (1-p\eta{\lambda}_n)r+ p\eta{\lambda}_n\frac{r}{4}+ p{\lambda}_n{\epsilon}_0M_0\\
 &\leq& (1-p\eta{\lambda}_n)r+\frac{p\eta {\lambda}_n}{4}r+\frac{p\eta {\lambda}_n}{4}r \\
 &= &\left(1-\frac{p\eta{\lambda}_n}{2}\right)r<r.
 \end{eqnarray*}
Hence, ${\phi}_p(x^*, x_{n+1}) \leq r.$ By induction, ${\phi}_p(x^*, x_n) \leq r  ~~ \forall  ~~ n\in \N.$ Thus, from inequality (\ref{e11}), $\left\{x_n\right\}$ is bounded.
\vskip 0.5 truecm

{\bf Part 2:} We now show that $\left\{x_n \right\}$ converges strongly to a solution of $Ax=0.$ Strongly monotone implies monotone, since $A$ is monotone and also satisfies the range condition and by the strict convexity of $E$, we obtain for every $t>0$, and $x\in E$, there exists a unique $x_t\in D(A)$, where $D(A)$ is the domain of $A$ such that
$$Jx\in Jx_t+tAx_t.$$
If $J_tx=x_t$, then we can define a single-valued mapping $J_t : E\rightarrow D(A)$ by $J_t=(J+tA)^{-1}J$. Such a $J_t$ is called the resolvent of $A$. Therefore, by Theorem \ref{t8}, for each $n\in \N$, there exists a unique $y_n\in D(A)$ such that

$$y_n=(J+\frac{1}{{\theta}_n}A)^{-1}Jx_1.$$

Then, we have $(J+\frac{1}{{\theta}_n}A)y_n=Jx_1$, such that
\begin{equation}\label{e24}
{\theta}_n(Jy_n-Jx_1)+Ay_n=0.
\end{equation} 
Observe that the sequence $\left\{y_n\right\}$ is bounded because it is a convergent sequence by Theorem \ref{t8}. Moreover, $\left\{x_n\right\}$ is bounded and hence $\left\{Ax_n\right\}$ is bounded. Following the same arguments as in part 1, we get,
\begin{eqnarray}\label{e15}
{\phi}_p(y_n, x_{n+1}) &\leq&  {\phi}_p(y_n, x_n )-p{\lambda}_n \left\langle x_n-y_n,Ax_n+{\theta}_n(Jx_n-Jx_1)\right\rangle+p{\lambda}_n{\psi}({\lambda}_nM_0)M_0\nonumber \\
&\leq & {\phi}_p(y_n, x_n )-p{\lambda}_n \left\langle x_n-y_n,Ax_n+{\theta}_n(Jx_n-Jx_1)\right\rangle+p{\lambda}_n{\epsilon}_0M_0.
\end{eqnarray}
By the strong monotonicity of $A$ and using Theorem \ref{t3} and Eq. (\ref{e24}), we obtain,
\begin{eqnarray*}
\left\langle x_n-y_n,Ax_n+{\theta}_n(Jx_n-Jx_1)\right\rangle
&= & \left\langle x_n-y_n,Ax_n+{\theta}_n(Jx_n-Jy_n+Jy_n-Jx_1)\right\rangle\\
&= & {\theta}_n\left\langle x_n-y_n,Jx_n-Jy_n\right\rangle + \left\langle x_n-y_n,Ax_n+{\theta}_n(Jy_n-Jx_1)\right\rangle\\
&= &{\theta}_n\left\langle x_n-y_n,Jx_n-Jy_n\right\rangle+\left\langle x_n-y_n,Ax_n-Ay_n \right\rangle\\
&\geq & {\theta}_ng(\|x_n-y_n \|) +\eta{\|x_n-y_n \|}^p \\
&\geq & \frac{1}{p}{\theta}_n {\phi}_p(y_n, x_n ) \ (\mbox{by Lemma \ref{l8} for some real constants} \ p > 1).
\end{eqnarray*}
Therefore, the inequality (\ref{e15}) becomes
\begin{eqnarray}\label{e25}
{\phi}_p(y_n, x_{n+1}) &\leq&  (1-{\lambda}_n{\theta}_n){\phi}_p(y_n, x_n ) +p{\lambda}_n{\epsilon}_0M_0.
\end{eqnarray}

Observe that by Lemma \ref{l31}, we have
\begin{eqnarray}\label{e26}
{\phi}_p(y_n, x_n )&\leq &  {\phi}_p(y_{n-1}, x_n )-p\left\langle y_n-x_n, Jy_{n-1}-Jy_n\right\rangle \nonumber \\ 
 &=& {\phi}_p(y_{n-1}, x_n )+p\left\langle x_n-y_n, Jy_{n-1}-Jy_n \right\rangle \nonumber \\
& \leq& {\phi}_p(y_{n-1}, x_n )+\|Jy_{n-1}-Jy_n\|\|x_n-y_n\|.
\end{eqnarray}

Let $R > 0$ such that $\|x_1\| \leq R, \|y_n\| \leq R$ for all $n \in  \N$. We obtain from Eq.(\ref{e24}) that
$$Jy_{n-1}-Jy_n+\frac{1}{{\theta}_n}\left(Ay_{n-1}-Ay_n\right)=  \frac{{\theta}_{n-1}-{\theta}_n}{{\theta}_n}\left(Jx_1-Jy_{n-1}\right).$$
 By taking the duality pairing of each side of this equation with respect to $y_{n-1}-y_n$ and by the strong monotonicity of $A$, we have
$$\left\langle Jy_{n-1}-Jy_n, y_{n-1}-y_n\right\rangle \leq  \frac{{\theta}_{n-1}-{\theta}_n}{{\theta}_n}\| Jx_1-Jy_{n-1}\|\| y_{n-1}-y_n\|,$$
which gives,
\begin{equation}\label{e27}
 \|Jy_{n-1}-Jy_n\| \leq \left( \frac{{\theta}_{n-1}}{{\theta}_n}-1\right)\|Jy_{n-1}-Jx_1\|.
\end{equation}
Using (\ref{e26}) and (\ref{e27}), the inequality (\ref{e25}) becomes
$${\phi}_p(y_n, x_{n+1}) \leq  (1-{\lambda}_n{\theta}_n){\phi}_p(y_{n-1}, x_n)+C\left( \frac{{\theta}_{n-1}}{{\theta}_n}-1\right)+ p{\lambda}_n{\epsilon}_0M_0,$$

for some constant $C > 0$.
By Lemma \ref{l11}, ${\phi}_p(y_{n-1}, x_n )\rightarrow 0$ as $n\rightarrow \infty$ and using Lemma \ref{l10}, we have that $x_n-y_{n-1}\rightarrow 0$  as $n\rightarrow \infty$.
Since by Theorem \ref{t8}, $y_n\rightarrow x^* \in N(A)$, we obtain that $x_n\rightarrow x^*$ as $n\rightarrow \infty$.

\end{proof}
\begin{corollary}\label{h5}
Let $E$ be a uniformly smooth and uniformly convex real Banach space and $E^*$ be its dual space. Suppose $A : E\rightarrow  E^*$ is a bounded and maximal monotone mapping such that $A^{-1}0\neq \emptyset$. Let  $\left\{ {\lambda}_n\right\}$ and  $\left\{ {\theta}_n\right\}$ be real sequences in $(0,1)$ such that,
\begin{itemize}
	\item[(i)] $\lim {\theta}_n =0$ and $\left\{ {\theta}_n\right\}$ is decreasing;
	\item[(ii)]$ \displaystyle\sum_{n=1}^{\infty} {\lambda}_n {\theta}_n=\infty$;
	\item[(iii)]$\displaystyle \lim_{n\rightarrow \infty}\left(({\theta}_{n-1}/{\theta}_n)-1\right)/{\lambda}_n {\theta}_n=0,~~ \displaystyle\sum_{n=1}^{\infty} {\lambda}^2_n <\infty$.
\end{itemize}	 

 For arbitrary $x_1 \in E$, define $\left\{x_n\right\}$ iteratively by:
 \begin{equation}
x_{n+1} = J^{-1}\left(Jx_n - {\lambda}_n\left(Ax_n+{\theta}_n(Jx_n-Jx_1)\right)\right), n \in \N,
\end{equation}
where $J$ is the generalized duality mapping from $E$ into $E^*$.
There exists a real constant ${\epsilon}_0>0$ such that ${\psi}({\lambda}_nM_0)\leq{\epsilon}_0, \ \ n \in \N$ for some constant $M_0>0$. Then, the sequence $\left\{x_n \right\}$ converges strongly to the solution of $Ax=0$.
\end{corollary}

\begin{proof}
Strong monotone implies monotone, therefore the result follows from Lemma \ref{h9} and by Theorem \ref{t6}.
\end{proof}
\begin{corollary}\label{h15}
Chidume and Idu \cite{b28}. Let $E$ be a uniformly convex and uniformly smooth real Banach space and $E^*$ be its dual space.  Let $T : E\rightarrow 2^{E^*}$ be a $J$-pseudocontractive and bounded map such that $(J-T)$ is maximal monotone. Suppose $ F^J_E(T)$=$\left\{ v \in E: Jv \in Tv \right\} \neq \emptyset$. For arbitrary $x_1, u \in E$, define a sequence $\left\{x_n\right\}$ iteratively by:

 \begin{equation}\label{h16}
x_{n+1} = J^{-1}\left((1-{\lambda}_n)Jx_n+{\lambda}_n{\eta}_n -{\lambda}_n{\theta}_n(Jx_n-Ju)\right), {\eta}_n\in Tx_n,~~ n \in \N,
\end{equation}
where $\left\{ {\lambda}_n\right\}$ and  $\left\{ {\theta}_n\right\}$ are real sequences in $(0,1)$ satisfying the following conditions:
\begin{itemize}
  \item[(i)]$ \displaystyle\sum_{n=1}^{\infty} {\lambda}_n {\theta}_n=\infty$,
	\item[(i)]${\lambda}_nM^*_0\leq {\gamma}_0{\theta}_n$; ${\delta}_E^{-1}({\lambda}_nM^*_0)\leq{\gamma}_0{\theta}_n$,
	\item[(iii)]$\frac{{\delta}_E^{-1}\left(\frac{{\theta}_{n-1}-{\theta}_n}{{\theta}_n} K\right)}{{{\lambda}_n \theta}_n}\rightarrow 0$; $\frac{{\delta}_{E^*}^{-1}\left(\frac{{\theta}_{n-1}-{\theta}_n}{{\theta}_n} K\right)}{{{\lambda}_n \theta}_n}\rightarrow 0$ as $n\rightarrow \infty$,
	\item[(iv)]$\frac{1}{2}\frac{{\theta}_{n-1}-{\theta}_n}{{\theta}_n} K \in (0, 1)$,
\end{itemize}
for some constants $M^*_0>0$ and ${\gamma}_0>0$, where ${\delta}_E:(0, \infty)\rightarrow (0, \infty)$ is the modulus of convexity of $E$ and $K:=4RL \sup\left\{\|Jx-Jy\|: \|x\|\leq R, \|y\|\leq R \right\}+1,~~ x, y \in E,~~ R>0$. Then the sequence $\left\{x_n\right\}$ converges strongly to a $J$-fixed point of $T$.
\end{corollary}

\begin{proof}
Define $A:= (J-T)$, then by the Lemma \ref{l15}, $A$ is a bounded maximal monotone map. Therefore, the iterative sequence (\ref{h16}) is equivalent to 
 \begin{equation}
x_{n+1} = J^{-1}\left(Jx_n - {\lambda}_n\left(Ax_n+{\theta}_n(Jx_n-Ju)\right)\right), n \in \N,
\end{equation}
where $J$ is the normalized duality mapping from $E$ into $E^*$. Hence, the result follows from the Corollary \ref{h5}.
\end{proof}

\begin{corollary}\label{h6}
Let $E$ be a uniformly smooth and uniformly convex real Banach space and $E^*$ be its dual space. Suppose $A : E\rightarrow  E^*$ is a semicontinuous bounded monotone mapping such that $A^{-1}0\neq \emptyset$. Let  $\left\{ {\lambda}_n\right\}$ and  $\left\{ {\theta}_n\right\}$ be real sequences in $(0,1)$ such that, 
\begin{itemize}
	\item[(i)] $\lim {\theta}_n =0$ and $\left\{ {\theta}_n\right\}$ is decreasing;
	\item[(ii)]$ \displaystyle\sum_{n=1}^{\infty} {\lambda}_n {\theta}_n=\infty$;
	\item[(iii)]$\displaystyle \lim_{n\rightarrow \infty}\left(({\theta}_{n-1}/{\theta}_n)-1\right)/{\lambda}_n {\theta}_n=0,~~ \displaystyle\sum_{n=1}^{\infty} {\lambda}^2_n <\infty$.
\end{itemize}			 
 
 For arbitrary $x_1 \in E$, define $\left\{x_n\right\}$ iteratively by:
\begin{equation}
x_{n+1} = J^{-1}\left(Jx_n - {\lambda}_n\left(Ax_n+{\theta}_n(Jx_n-Jx_1)\right)\right), n \in \N,
\end{equation}
where $J$ is the generalized duality mapping from $E$ into $E^*$.
There exists a real constant ${\epsilon}_0>0$ such that ${\psi}({\lambda}_nM_0)\leq{\epsilon}_0, \ \ n \in \N$ for some constant $M_0>0$. Then, the sequence $\left\{x_n \right\}$ converges strongly to the solution of $Ax=0$.  
\end{corollary}

\begin{proof}
The result follows from Lemma \ref{h4} and by the Corollary \ref{h5}.
\end{proof}

\begin{corollary}\label{h7}
Let $E$ be a uniformly smooth and uniformly convex real Banach space and $E^*$ be its dual space. Suppose $A : E\rightarrow  E^*$ is a uniformly continuous and maximal monotone mapping such that $A^{-1}0\neq \emptyset$. Let  $\left\{ {\lambda}_n\right\}$ and  $\left\{ {\theta}_n\right\}$ be real sequences in $(0,1)$ such that,
\begin{itemize}
	\item[(i)] $\lim {\theta}_n =0$ and $\left\{ {\theta}_n\right\}$ is decreasing;
	\item[(ii)]$ \displaystyle\sum_{n=1}^{\infty} {\lambda}_n {\theta}_n=\infty$;
	\item[(iii)]$\displaystyle \lim_{n\rightarrow \infty}\left(({\theta}_{n-1}/{\theta}_n)-1\right)/{\lambda}_n {\theta}_n=0,~~ \displaystyle\sum_{n=1}^{\infty} {\lambda}^2_n <\infty$.
\end{itemize}	   
 For arbitrary $x_1 \in E$, define $\left\{x_n\right\}$ iteratively by:
 \begin{equation}
x_{n+1} = J^{-1}\left(Jx_n - {\lambda}_n\left(Ax_n+{\theta}_n(Jx_n-Jx_1)\right)\right), n \in \N,
\end{equation}
where $J$ is the generalized duality mapping from $E$ into $E^*$.
There exists a real constant ${\epsilon}_0>0$ such that ${\psi}({\lambda}_nM_0)\leq{\epsilon}_0, \ \ n \in \N$ for some constant $M_0>0$. Then, the sequence $\left\{x_n \right\}$ converges strongly to the solution of $Ax=0$.
\end{corollary}

\begin{proof}
The result follows from Lemma \ref{l25} and by the Corollary \ref{h5}.
\end{proof}

\begin{corollary}\label{t4}
Aibinu and Mewomo \cite{b9}. Let $E$ be a $p$-uniformly convex real Banach space with uniformly G$\hat{a}$teaux differentiable norm such that $\frac{1}{p}+\frac{1}{q}=1, p\geq 2$ and $E^*$ its dual space. Let $A : E\rightarrow  E^*$ be a bounded and $\eta$-strongly monotone mapping such that $A^{-1}0\neq \emptyset$. For arbitrary $x_1 \in E$, let $\left\{x_n\right\}$ be the sequence defined iteratively by
\begin{equation}
x_{n+1} = J^{-1}(Jx_n - {\lambda}_nAx_n),  n \in \N,
\end{equation}
where $J$ is the generalized duality mapping from $E$ into $E^*$ and $\left\{ {\lambda}_n\right\} \subset (0, {\gamma}_0), {\gamma}_0 \leq 1$ is a real sequence satisfying the following conditions:
\begin{itemize}
	\item[(i)] $\displaystyle\sum_{n=1}^{\infty} {\lambda}_n = \infty$; 
	\item[(ii)]$ \displaystyle\sum_{n=1}^{\infty} {\lambda}^2_n <\infty$.
\end{itemize}	
Then, the sequence $\left\{x_n \right\}$ converges strongly to the unique point $x^* \in A^{-1}0$.
\end{corollary}
\begin{proof}
  By taking ${\theta}_n=0$ in Theorem \ref{t6}, we obtain the desired result.
\end{proof}

\begin{corollary}
Diop et al. \cite{b13}.  Let $E$ be a $2$-uniformly convex real Banach space with uniformly G$\hat{a}$teaux differentiable norm and $E^*$ its dual space. Let $A : E\rightarrow  E^*$ be a bounded and $k$-strongly monotone mapping such that $ A^{-1}0\neq \emptyset $. For arbitrary $x_1 \in E$, let $\left\{x_n\right\}$ be the sequence defined iteratively by:
\begin{equation}
x_{n+1} = J^{-1}(Jx_n - {\alpha}_nAx_n), n\in \N,
\end{equation}
where $J$ is the normalized duality mapping from $E$ into $E^*$ and $\left\{ a_n\right\} \subset (0, 1)$ is a real sequence satisfying the following conditions:
\begin{itemize}
\item[(i)] $\displaystyle\sum_{n=1}^{\infty} {\alpha}_n = \infty$;
\item[(ii)]$\displaystyle\sum_{n=1}^{\infty} {\alpha}^2_n <\infty$.
\end{itemize}
Then, there exists ${\gamma}_0 > 0$ such that if $ {\alpha}_n < {\gamma}_0 $, the sequence $\left\{x_n \right\}$ converges strongly to the unique solution of the equation $Ax = 0$.
\end{corollary}
\begin{proof}
  By taking $p=2$ in Corollary \ref{t4}, we obtain the desired result.
\end{proof}
\begin{corollary}
Chidume and Djitte \cite{b6}: Let $E$ be a $2$-uniformly smooth real Banach space, and let $A : E \rightarrow E$ be a bounded $m$-accretive mapping. For arbitrary $x_1\in E$, define the sequence $\left\{x_n\right\}$ iteratively by 
$$x_{n+1}:=x_n-{\lambda}_nAx_n-{\lambda}_n{\theta}_n(x_n-x_1),  ~~n\in \N, $$
where $\left\{{\lambda}_n\right\}$ and $\left\{{\theta}_n\right\}$ are sequences in $(0,1)$ satisfying the conditions:
\begin{itemize}
	\item[(i)] $\lim {\theta}_n =0$ and $\left\{ {\theta}_n\right\}$ is decreasing;
	\item[(ii)]$ \displaystyle\sum_{n=1}^{\infty} {\lambda}_n {\theta}_n=\infty$, ${\lambda}_n=o({\theta}_n)$;
	\item[(ii)]$\displaystyle \lim_{n\rightarrow \infty}\left(({\theta}_{n-1}/{\theta}_n)-1\right)/{\lambda}_n {\theta}_n=0,~~ \displaystyle\sum_{n=1}^{\infty} {\lambda}^2_n <\infty$.
\end{itemize}	
Suppose that the equation $Ax=0$ has a solution. Then, there exists a constant ${\gamma}_0 > 0$ such that if ${\lambda}_n\leq {\gamma}_0 {\theta}_n$ for all $n\in \N$, $\left\{x_n\right\}$ converges strongly to a solution of the equation $Ax=0$.
\end{corollary}
\begin{proof}
The result follows from the Theorem \ref{t6} since uniformly smooth and uniformly convex spaces are more general.
\end{proof}

\begin{remark}
The Lyapunov functions which we introduced admit the generalized duality mapping. Therefore, the duality mapping, $J$ in our iteration is a generalized one while in Chidume and Idu \cite{b28}, $J$ is the normalized duality mapping. Clearly, our results show the efficacy of the new geometric properties in Banach spaces. The iterative algorithm study by Chidume and Djitte \cite{b6} has been successfully extended into uniformly smooth and uniformly convex Banach spaces for strongly monotone mappings. Also, our method of prove is constructive and is of independent interest. 
\end{remark}
\begin{remark}
 Prototype for our iteration parameters in Theorem \ref{t6} are, ${\lambda}_n=\frac{1}{(n+1)^a}$ and ${\theta}_n=\frac{1}{(n+1)^b}$, where $0 < b < a$ and $a + b < 1$.
\end{remark}

\vskip 0.5 truecm
\textbf{Acknowledgement:\\} The first author acknowledge with thanks the bursary and financial support from Department of Science and Technology and National Research Foundation, Republic of South Africa Center of Excellence in Mathematical and Statistical Sciences (DST-NRF CoE-MaSS) Doctoral Bursary. Opinions expressed and conclusions arrived at are those of the authors and are not necessarily to be attributed to the CoE-MaSS.

\end{document}